\documentclass[letterpaper, 10 pt, conference]{ieeeconf}  

\usepackage{amsmath}
\usepackage{amssymb}
\usepackage{color}
\usepackage{enumerate}
\usepackage{lipsum}
\usepackage{mathtools}
\usepackage{cuted}
\usepackage{tabu}
\usepackage{float}

\usepackage{nameref}
\usepackage{varioref}
\usepackage{hyperref}
\usepackage{cleveref}

\usepackage{amsmath}                                   
\usepackage{amssymb} 

\usepackage{rotating}

\usepackage{graphicx}
\usepackage{subcaption}

\usepackage{algorithm}
\usepackage{algpseudocode}
\algnewcommand{\Initialize}[1]{%
	\State \textbf{Initialization:}
	\Statex {\raggedright #1}
}

\newcommand{\proj}[2][] {{\mathcal{P}}_{{#1}} {\left(#2\right)}}
\newcommand{\EXP}[1]{\mathsf{E}\!\left[#1\right] }
\newcommand{\prob}[1]{\mathsf{p}_#1} 

\newcommand{\us}[1]{{\color{black}#1}}

\newtheorem{assumption}{Assumption}
\newtheorem{remark}{Remark}
\newtheorem{theorem}{Theorem}
\newtheorem{lemma}{Lemma}
\newtheorem{proposition}{Proposition}
\newtheorem{definition}{Definition}
\newtheorem{corollary}{Corollary}

\def\argmin{\mathop{\rm argmin}}

\IEEEoverridecommandlockouts                              

\title{\LARGE \bf
A Randomized Block Coordinate Iterative Regularized Gradient Method for  High-dimensional Ill-posed Convex Optimization}

\author{Harshal Kaushik$^{1}$ and Farzad Yousefian$^{2}$
\thanks{$^{1}$Harshal Kaushik is a graduate student at the  Department of Industrial Engineering and Management, Oklahoma State University, Stillwater, OK 74074, USA.
{\tt\small harshal.kaushik@okstate.edu}}%
\thanks{$^{2}$Farzad Yousefian is an assistant professor at the Department of Industrial Engineering and Management, Oklahoma State University, Stillwater, OK 74074, USA.
{\tt\small farzad.yousefian@okstate.edu}}%
}
\usepackage{hyperref}
\usepackage{cite}
\begin{document}

\maketitle
\thispagestyle{empty}
\pagestyle{empty}

\begin{abstract}
Motivated by  high-dimensional nonlinear optimization problems as well as ill-posed optimization problems arising in image processing, we consider a bilevel optimization model where we seek among the optimal solutions of the inner level problem, a solution that minimizes a secondary metric. Our goal is to address the high-dimensionality of the bilevel problem, and the nondifferentiability of the objective function. Minimal norm gradient, sequential averaging, and iterative regularization are some of the recent schemes developed for addressing the bilevel problem. But none of them address the high-dimensional structure and nondifferentiability. With this gap in the literature, we develop a randomized block coordinate iterative regularized gradient descent scheme \eqref{algorithm_proposed}. We establish the convergence of the sequence generated by \ref{algorithm_proposed} to the unique solution of the bilevel problem of interest. Furthermore, we derive a rate of convergence  $ \centering \us{\cal O} \left(\frac{1}{{k}^{0.5-\delta}}\right)$, with respect to the inner level objective function. We demonstrate the performance of \ref{algorithm_proposed}  in solving the ill-posed problems arising in image processing.

\end{abstract}

\section{Introduction}



In this work we are interested in solving a bilevel problem given as,
	\begin{equation} \label{Q}
		\tag{$P^g_f$}  
		\begin{split}
			& \text{minimize} \  g( x) \\
			 &\textrm {s.t.} \quad  x \in    \argmin\{f(x): x \in X \},
		\end{split}
	\end{equation}	
	where functions $f$ and $g$ are defined as $ f: \mathbb{R}^n\rightarrow\mathbb{R}$ and $g: \mathbb{R}^n\rightarrow\mathbb{R}$. \eqref{Q} is a high-dimensional structure in a sense that the dimensions of the solution space can be huge. This causes high computation efforts in taking the gradient at any iteration. Set $X$ is assumed to be having a block structure, i.e. it can be written as, 
		$\displaystyle X =  \prod_{i = 1}^{d}  X_i,$
	 where  $ X_i \subseteq \mathbb{R}^{n_i} $ and $\sum_{i = 1}^{d} n_i  =  n$. Precisely, the assumptions on \eqref{Q} are provided next. 
	\begin{assumption}  \label{obj_func}
	Let the following hold: 
	\begin{itemize}
		\item[(a)] Any block $i$ of set $X$ $ (X_i \subseteq \mathbb{R}^{n_i})$ is assumed to be nonempty, closed, and convex for all $ \ i = 1, \dots, d$.
		\item[(b)] $ f: \mathbb{R}^n\rightarrow (-\infty, \infty]$ is a nondifferentiable, proper, and convex function. 
		\item[(c)] $g:\mathbb{R}^n\rightarrow(-\infty, \infty]$ is a nondifferentiable, proper, and $\mu$-strongly convex function $(\mu > 0)$. 
		\item[(d)] $X \subseteq \mathrm{int}\left(\mathrm{dom}( f) \cap \mathrm{dom}(g)\right)$.
	\end{itemize}
\end{assumption}
\subsection{Motivating examples}
Here we present two applications of the formulation \eqref{Q}. 

\begin{table*}[t]
	\caption{Comparison of schemes for solving bilevel optimization problem}
	\centering
	\begin{tabular}{c | c | c | c | c | c | c}
		\hline
		\begin{minipage}{0.4cm} Ref. \end{minipage} & Problem formulation & Assumption & Scheme & Scale & Metric & Rate \\ 
		\hline
		\begin{minipage}{0.4cm} \centering \cite{Solodov_2007} \end{minipage} & {$ \displaystyle\begin{aligned} &\text{minimize} \  g( x) \\ \textrm {s.t.} \quad  x \in &   \argmin\{f(x): x \in X \}  \end{aligned}$}& \begin{minipage}{2.5cm}$f$ and $g$ both smooth and convex\end{minipage} & \begin{minipage}{2.0cm} \centering Iterative regularized\end{minipage} & \begin{minipage}{01cm} \centering Standard scale \end{minipage} & $-$ & $-$  \\
		\hline 
		\begin{minipage} {0.4cm}\centering\cite{BeckSabach_2014}\end{minipage}  & {$ \displaystyle\begin{aligned} &\text{minimize} \  g( x) \\	\textrm {s.t.} \quad  x \in  &  \argmin\{f(x): x \in X \}  \end{aligned}$} & \begin{minipage}{2.5cm}$f$ convex, Lipschitz cont. $g$ strongly conv. \end{minipage}  & \begin{minipage}{2.0cm} \centering Minimal norm gradient \end{minipage} & \begin{minipage}{01cm}\centering Standard scale \end{minipage} & \begin{minipage}{0.5cm}Inner level\end{minipage}  & $\centering\us{\cal O} \left(\frac{1}{\sqrt{k}}\right)$\\
		\hline
		\begin{minipage} {0.4cm}\centering\cite{SabachShtern_2017} \end{minipage}  &  {$ \displaystyle\begin{aligned} &\text{minimize} \  g( x) \\	x \in    \argmin&\{f_1(x)+f_2(x): x \in \mathbb{R}^n \}  \end{aligned}$}  & \begin{minipage}{2.5cm}$f_1, f_2$ convex. $g$ is strongly  convex. $f_1$, $g$  Lipschitz cont. \end{minipage} & \begin{minipage}{2.0cm} \centering Sequential averaging \end{minipage} & \begin{minipage}{01cm} \centering Standard scale \end{minipage} & \begin{minipage}{0.5cm}Inner level\end{minipage} & $\us{\cal O} \left(\frac{1}{k}\right)$\\
		\hline
		\begin{minipage} {0.4cm}\centering\cite{Yousefian} \end{minipage} &  {$  \begin{aligned}  &  \text{minimize }\|x\| \\  \text{s. t. }x\in \text{SOL}&(X,F), \text{ where } F(x) \triangleq f(x,\xi) \end{aligned}$}  & \begin{minipage}{2.5cm} $F$ is monotone and continuous. \end{minipage}   & \begin{minipage}{2.0cm} \centering  $\text{aRSSA}_{\textit{l,r}}$ \end{minipage} & \begin{minipage}{01cm}\centering Standard scale \end{minipage} & \begin{minipage}{0.5cm} Outer \\ level \end{minipage} & {$ \displaystyle \begin{aligned} \us{\cal O} \left(\frac{1}{{k}^{1/6-\delta}}\right) \end{aligned} $} \\
		\hline
		\begin{minipage} {0.4cm}\centering\cite{Ozdaglar} \end{minipage}  & \begin{minipage}{4cm} \vspace*{-0.19cm} \begin{align*} 	 \min_{x_i \in X_i,  z\in Z} \textstyle  \sum_{i=1}^{N}f_i\left(x_i\right) \	\textrm { s.t. } Dx + Hz = 0 \end{align*} \vspace*{-0.25cm} \end{minipage} & \begin{minipage}{2.5cm}  $f_i$ is convex and possibly nonsmooth. \end{minipage} & \begin{minipage}{2.0cm}\centering Asynchronous ADMM \end{minipage} & \begin{minipage}{01cm} \centering Standard Scale	\end{minipage} &  Feasibility & $\us{\cal O} \left(\frac{1}{{k}}\right)$\\
		\hline
		\begin{minipage} {0.4cm} \centering\cite{Aybat_Linear_ADMM}  \end{minipage} &{$   \begin{aligned}  &\textstyle \text{for}  \ \mathcal{G}(\mathcal{N},\mathcal{E}),  \text{ min} \textstyle  \sum_{i\in \mathcal{N}}\xi_i( x)+f_i(x) \\	&\textrm {s.t.} \ x \in \mathbb{R}^n, \  x_i=x_j \ \text{ for all }(i,j)\in\mathcal{E}\end{aligned}$} &\begin{minipage}{2.5cm}  $\xi_i, f_i $ are convex, $f_i$  Lipschitz continuous. \end{minipage}  & \begin{minipage}{2.0cm}\centering \begin{minipage}{2cm} Distributed proximal gradient 
		\end{minipage} \end{minipage} & \begin{minipage}{01.2cm}\centering \centering Standard Scale	\end{minipage} &  Feasibility  & {$\begin{aligned}  \us{\cal O} \left({1}/{{k}}\right)  \end{aligned} $} \\
		\hline
		\begin{minipage} {0.4cm} \centering\cite{Xu_primalDualPriximal} \end{minipage}  &  {$ \displaystyle\begin{aligned} \min_x & \  g_1( x) +g_2(x)\\	&\textrm {s.t. } Ax=b \end{aligned}$} & \begin{minipage}{2.5cm}$g_1$ is convex and Lipschitz continuous. $g_2$ is convex. \end{minipage} & \begin{minipage}{2.0cm} \centering  Linear \\ ADMM \end{minipage} & \begin{minipage}{1cm} \centering Standard Scale	\end{minipage} & \begin{minipage}{01.15cm}Feasibility \\ Optimality\end{minipage} & {$ \begin{aligned} \centering   \us{\cal O} \left({1}/{{k}^2}\right) \end{aligned} $} \\
		\hline
		\begin{minipage} {0.4cm}{\bf This}\\ {\bf work}\end{minipage}  &  { $\begin{aligned} &\text{minimize} \  g( x) \\ \vspace{-0.5cm}	    x \in    \argmin&\{f(x): x \in X\}; \ X =  {\textstyle\prod_{i = 1}^{d}  X_i}  \end{aligned}$} & \begin{minipage}{2.5cm} $f$ is convex and $g$ is strongly convex. \end{minipage}& \begin{minipage}{2.2cm}  Random block iterative regularized gradient \end{minipage} & \begin{minipage}{0.8cm} \begin{minipage}{0.8cm} \centering Large Scale	\end{minipage} \end{minipage} &Feasibility  &  $ \centering \us{\cal O} \left(\frac{1}{{k}^{0.5-\delta}}\right)$ \\
		\hline
	\end{tabular}
\label{table1}
\vspace*{-0.40cm}
\end{table*}

\noindent (i) {\bf High-dimensional nonlinear constrained optimization}: Consider the following  problem with nonlinear constraints,  
\begin{equation} \label{Q_h}
\tag{$1$}  
\begin{split}
& \text{minimize} \  g( x) \\
&\textrm {s.t.} \ h_i(x) \leq 0 \ \text{ for } i = 1, \dots, m\\
&x \in X \triangleq  \prod_{i = 1}^{d}  X_i.
\end{split}
\end{equation}
The assumptions on \eqref{Q_h} are: (i) function $h_i:\mathbb{R}^n\rightarrow\mathbb{R}$ is convex; (ii) function $g:\mathbb{R}^n\rightarrow\mathbb{R}$ is convex; (iii) $X \subseteq \mathbb{R}^n$ is convex, satisfying Assumption \ref{obj_func} (a).

 Provided that the feasible set of \eqref{Q_h} is nonempty, this problem  can be equivalently written in a bilevel structure as following,
\begin{equation*} 
\begin{split}
& \text{minimize} \  g( x) \\
\textrm {s.t.} \quad  x \in  \argmin& \Bigg \{f(x)\triangleq\sum_{i = 1}^{m}\max\big\{0,h_i(x)\big\} : x \in X  \Bigg \}  .
\end{split}
\end{equation*}	
\noindent (ii) {\bf Ill-posed optimization}: Linear inverse problems arising in image deblurring can be written as the following optimization problem,
\begin{equation*} \label{ImageProc}
\tag{2}
\begin{split}
& \text{minimize} \  \left\Vert\left. Ax - b\right\Vert\right.^2\\
&\textrm {s.t.} \quad  x \in \mathbb{R}^n,
\end{split}
\end{equation*}	
where
$A$ is a blurring operator  $(A \in \mathbb{R}^{m\times n}), $
$b$ is the given blurred image  $ (b \in \mathbb{R}^{m}),$ and
$x$ is a deblurred image  $(x \in \mathbb{R}^{n}).$ This is an ill-posed problem in a sense that there may be multiple solutions or the optimal solution $x$ may be very sensitive to the perturbation in the input $b$. To address the ill-posedness, problem \eqref{ImageProc} can be reformulated in a bilevel structure as following, (see \cite{Rosasco_2018}). 
\begin{equation*} 
\begin{split}
& \text{minimize} \ \|x\|^2 \\
&\textrm {s.t.} \quad  x \in \argmin\Big\{\left\Vert\left. Ax - b\right\Vert\right.^2 : x \in \mathbb{R}^n \Big\}.
\end{split}
\end{equation*}

\subsection{Existing methods}
Sequential regularization, minimal norm gradient, sequential averaging, and iterative regularization are some of the recent schemes developed for addressing problem \eqref{Q}. One of the classical approaches to address the ill-posedness is the regularization technique. 
\begin{equation} \label{eta_problem}
\tag{$P_{\eta}$}  
\begin{split}
& \text{minimize} \quad  f ( x) + \eta  g( x) \\
& \textrm { s.t.} \quad  x \in  X.
\end{split}
\end{equation}
Tikhonov in \cite{Tikhonov_1977} showed that under some assumptions, the solution of regularized problem \eqref{eta_problem} converges to the solution of the inner level problem of \eqref{Q} as the regularization parameter $\eta$ goes to zero. Later the threshold value of $\eta$, under which the solution of \eqref{eta_problem} is same as the solution of the inner level problem of \eqref{Q} was studied under the area of {\it exact regularization} \cite{Manga_1979, Manga_1985, Tseng_2007}. There have been numerous theoretical studies in the 80's, 90's \cite{Bertsekas_1975,Burke_1991,Manga_1979, Bertsekas_1982, Fletcher_1984, Manga_1985} and early 2000 \cite{Bert_Nedich_Ozzi_2003, Conn_2000} on finding the suitable $\eta$, but in practice there is not much guidance on tuning this parameter. Finding a suitable $\eta$ necessitate solving a sequence of problem \eqref{eta_problem} for ${\eta_k}$, where $\eta_{{k}}\rightarrow 0$. This {\it two loop scheme} is highly inefficient, especially in high dimensional spaces. 

In the past decade, interest has been shifted to solving  the bilevel problem \eqref{Q} using  {\it single loop schemes}. Solodov in \cite{Solodov_2007} showed that for both functions $g$ and $f$ in \eqref{Q} with Lipschitz gradient, and  $f$ to be a composite function with the indicator function, solutions to \eqref{Q} can be found by  iterative regularized gradient descent with sequence  $\eta_k\rightarrow0$ and $\sum_{k=1}^{\infty} \eta_k = \infty$. In \eqref{eta_problem}, when $g$ is $\ell_2$ norm in variational inequality regimes, Yousefian et al  showed that solution to \eqref{Q} can be found by employing an iterative regularized gradient descent scheme (see \cite{Yousefian}).

In 2014, minimal norm gradient (MNG) scheme was proposed  \cite{BeckSabach_2014}. This involves solving the projection (this itself is an another optimization problem) for each iteration $k$, which makes MNG to be difficult to implement for the large scale problems. Later in \cite{SabachShtern_2017} a sequential averaging scheme (BiG-SAM) was developed  with a  rate of convergence $\us{\cal O}\left(1/k\right)$.  Recently in \cite{Rosasco_2018} a general iterative regularized algorithm based on a primal-dual  diagonal descent method was proposed to solve \eqref{Q}.  

In these  papers, the missing part is addressing the high-dimensional structure, which is  common in the high resolution image processing problems. Our goal is to bridge this gap by developing a randomized block coordinate iterative regularized gradient descent scheme to solve the high-dimensional problems.  

Coordinate descent methods have recently gained popularity due to their potential of solving the large-scale optimization problems. In \cite{Takac_2014, Nesterov_2012}, block coordinate descent found to be effective when the size of solution space is of the order $10^{8}-10^{12}$.  Therefore block strategy is effective when dealing with high dimensionality. Cyclic coordinate descent is a common strategy to  make the selection of block. It is well studied in the past but recently the focus has been shifted to randomized strategy  due to theoretical \cite{Nesterov_2012,Takac_2014,ShaiZhang_2013} and practical advantages it offers in solving the large scale machine learning problems \cite{ChangLin_2008,ChangLinKeerthi_2008,Tewari_2009,ShaiZhang_2013}.   

High-dimensional nonlinear constrained optimization \eqref{Q_h} is the another problem we consider in this work. One of the popular primal-dual methods is Alternating Direction Method of Multipliers (ADMM) \cite{Ozdaglar,Aybat_Linear_ADMM,Xu_primalDualPriximal}. One of the underlying assumptions for ADMM is the linear constraints. In our work, bilevel problem \eqref{Q_h} addresses the nonlinearity in the constraints and  the high-dimensionality of the space. 
\subsection{Main contributions}
(I) We develop a single loop first order scheme \ref{algorithm_proposed} with the mild requirements such as  $f$ and $g$ can be nondifferentiable functions.  (II) \ref{algorithm_proposed} can handle the high-dimensional structure of bilevel problem \eqref{Q}.  (III) We establish the convergence of the sequence generated from \ref{algorithm_proposed} to the unique solution of \eqref{Q}. (IV) We derive the rate of convergence  $ \centering \us{\cal O} \left(\frac{1}{{k}^{0.5-\delta}}\right)$, with respect to the inner level function of the bilevel problem.  \\To highlight the contribution of our work and its distinction from the other methods, we provide a table (see TABLE \ref{table1}).

The rest of the paper is organized as follows. Section II, we propose \ref{algorithm_proposed} scheme with preliminaries. Section III is for showing the convergence of \ref{algorithm_proposed} to the solution of bilevel problem \eqref{Q}. In Section IV, we show the rate analysis of \ref{algorithm_proposed} with respect to the inner level objective function of \eqref{Q}. In Section V, we apply \ref{algorithm_proposed} to image deblurring application and discuss the computational effectiveness of our scheme. In Section VI, we  highlight the main contribution and provide the concluding remarks.\\
\noindent{\bf Notation: }Vector $ x$ is assumed to be a column vector ($ x \in \mathbb{R}^n $), $ x^T$ is the transpose. $x^{(i)}$ denotes  the $i^{\text{th}}$  block of dimensions for a vector $x$. $X_i$ denotes the  $i^{\text{th}}$ block of dimensions for set $X$. $\| x\|$ denotes the Euclidean vector norm, i.e., $\| x\|=\sqrt{ x^T  x}.$  $ \proj[S]{ s} $ is used for the Euclidean projection of  vector $ s$ on a set $ S$, i.e., $\left\Vert\left.  s- \proj[ S]{ s}\right\Vert\right.= \min_{ y \in  S}\| s- y\|.$ a.s. used for 'almost surely'.  $\mathcal{F}_k$ denotes the set of variables $\{ i_0, \dots, i_{k-1} \}$. For a random variable $i_k$,  $\mathsf{Prob}(i_k = i)$ is $ \prob{{i_k}}$. $\tilde{\nabla}$ denotes the subgradient and $\delta$ denotes the subdifferential set.  $\tilde{\nabla}_{i} f(x)$ is the $i^{\text{th}}$ block of $\tilde{\nabla} f(x)$. $\min_{1\leq i \leq d}\{\prob{i}\}$ is denoted by $\prob{{min}}$ and $\max_{1\leq i \leq d}\{\prob{i}\}$ is denoted by $\prob{{max}}$.
 \section{Algorithm outline} \label{algo_outline}
Here we explain algorithm \ref{algorithm_proposed} the required preliminaries for convergence and rate analysis. 
    \subsection{Proposed scheme \ref{algorithm_proposed}}
  Here, a randomized block coordinate iterative regularized gradient descent scheme \eqref{algorithm_proposed} is proposed for solving \eqref{Q}. In \ref{algorithm_proposed}, both the sequences of regularization parameter $\eta_k$ and stepsize parameter  $\gamma_k$ are in terms of iteration $k$.  To address the high-dimensionality, at each iteration we update a random block of the iterate  $x_k$. Selection of block $i_k$ at iteration $k$ is governed by  Assumption \ref{random sample}.  Finally, averaging is employed which will be helpful in deriving the rate statement.
%
 
 
 \begin{algorithm}  
 	\caption{Randomized block iterative regularized gradient descent (RB-IRG) algorithm}
 	\begin{algorithmic}[1]
 		\Initialize{Set k = 0, select a point $x_0 \in X$, parameters $\gamma_{_0} > 0$, and $\eta_{_0} > 0, {S_0 = \gamma_{_0}^r, \text{ and } \bar{x}_0 = x_0 }$.  }\label{algorithm_1}
 		\For{k = 0, 1, \dots, {N-1}}
 		\State {$i_k$ is generated by Assumption \ref{random sample}}.
 		\State {{Compute $ \tilde{\nabla}_i f(x_k ) \in \partial f\left(x_k^{(i)}\right) $  and $ \tilde{\nabla}_i g(x_k ) \in \partial g\left(x_k^{(i)}\right) \text{ for } x_k^{(i)} \in X_i$}.} 
 		\State Update { $x_{k+1}^{i_k}$}:=
 		 \begin{gather*}\label{algorithm_proposed}\tag{RB-IRG}
 		 \begin{cases}
 		\proj[X_{i}]{ x_{k}^{(i)}-\gamma_{k}\left(\tilde{\nabla}_{i} f\left( x_{k}\right) + \eta_{k} \tilde{\nabla}_{i} g\left( x_{k}\right)\right)}\quad \text{if } i = i_k . \nonumber\\
 		x_k^{(i)} \quad \text{if } i \neq i_k.
 		\end{cases}
 	\end{gather*}
 		\State Update $\overline{x}_{k}$ as following,
 		\begin{align}\label{alg_avg}
 		S_{k+1} &= S_k + \gamma_{k+1}^r,\quad	\overline{x}_{k+1} = \frac{S_k\overline{x}_{k}+\gamma_{k+1}^rx_{k+1}}{S_{k+1}}.
 		\end{align}
 		\EndFor
 	\end{algorithmic}
 \end{algorithm}
 
 \begin{assumption}{\bf(Random sample $i_k$)} \label{random sample}
 	Random variable $i_k$ is generated at each iteration $k$ from an i.i.d. distribution governed by probability $\prob{{i_k}}$ where prob{{$(i_k=i)$}} = $\prob{{i_k}}>0$, and $\sum_{i=1}^{d}\prob{{i_k}}=1.$ 
 \end{assumption}
 
\subsection{Preliminaries} 
 Throughout the paper, we use $x^*_g$ and $x^*_{\eta_{k}}$ to denote the unique minimizers of \eqref{Q} and \eqref{eta_problem} respectively.
%
\begin{remark}
From Assumptions \ref{obj_func} (b, c), the objective function of (\ref{eta_problem}), is a strongly convex. The feasible region of (\ref{eta_problem}) is closed and convex (from Assumption \ref{obj_func}(a)). Therefore (\ref{eta_problem}) has a unique minimizer. (cf. Ch. 2 of \cite{Pang_2003}). 
%
%
%
Similarly, we can claim that \eqref{Q} has a unique minimizer. 
\end{remark}

 \begin{remark}\label{convexity_subgradients}
 	In problem (\ref{Q}), for any $x_1, x_2 \in X$, for a  convex function $f$ and $\mu$-strongly convex function $g$,\\ $
 	\left(\tilde{\nabla}f(x_1)-\tilde{\nabla}f(x_2)\right)^T\left( x_1 - x_2 \right) \geq 0, \\ \left(\tilde{\nabla}g(x_1)-\tilde{\nabla}g(x_2)\right)^T\left( x_1 - x_2 \right) \geq \mu \|x_1-x_2\|^2.$
 \end{remark}
%
The following lemma is used  in proving the convergence. 
 \begin{lemma} {\bf (Lemma 10, pg. 49 of \cite{Polyak_1987}): } \label{Lemma_convergence_2} 
 	Let $\{v_k\}$ be a sequence of nonnegative random variables, where $E[v_0]<\infty$,  and let $\{\alpha_k\}$ and $\{\beta_k\}$ be deterministic scalar sequences such that: 
 	$\EXP{v_{k+1}|v_0, \textrm{ }..., v_k} \leq (1-\alpha_k)v_k + \beta_k$ for all $ k \geq 0,$ $ \,	0 \leq \alpha_k \leq 1, $ $ \beta_k \geq 0,$ $  \sum_{k=0}^{\infty} \alpha_k = \infty, $ $ \sum_{k=0}^{\infty} \beta_k < \infty, $ $\lim_{k\to\infty}  \frac{\beta_k}{\alpha_k} = 0.$ 	
 	Then, $v_k \rightarrow 0$, a.s., and { $ \lim\limits_{k\rightarrow \infty} \EXP{v_k} = 0$}. 
 \end{lemma}
The next result will be used in our analysis.
 \begin{lemma}{\bf (Theorem 6, pg. 75 of \cite{Knopp_1951}):} \label{convergence_sum}
	Let $\{u_t\}$  $\left(\subset \mathbb{R}^n\right)$ be a convergent sequence such that it has a limit point $\hat{u}\in \mathbb{R}^n$ and consider another sequence $\{\alpha_k\}$ of positive numbers such that $\sum_{k=0}^{\infty} \alpha_k = \infty$. Suppose $v_k$ is given by $ v_k = \frac{\sum_{t=0}^{k-1}\left(\alpha_t u_t\right)}{\sum_{t=0}^{k-1}\alpha_t}$, for all $k\geq 1$. Then $ \lim\limits_{k\rightarrow\infty}v_k = \hat{u}$.
\end{lemma}
\begin{remark} \label{remark_subgrad}
	From Assumption \ref{obj_func} (b, c, d),  for all  $\ x \in X$,  the set $\partial f(x) $ is nonempty and bounded (cf. Ch. 3 of  ~\cite{AmirBeck_book_2017}).  Similarly  $\partial g(x) $ is nonempty and bounded for all $x \in X$.
\end{remark}
\begin{remark} \label{remark_subgrad_bound}
	From Remark \ref{remark_subgrad}, let us say that for any $x^{(i)}  \in X_i $, there exists a scalar $C_{f,i}$ such that  $\left\Vert\left.\tilde{\nabla}_i f(x)\right\Vert\right.\leq C_{f,i}$. Let $ C_f \triangleq \sqrt{\sum_{i=1}^{d} C_{f,i}^2}$. Now we have,
	$\left\Vert\left. \tilde{\nabla} f(x) \right\Vert\right. \leq C_{f}$ for all $ x \in X.$ Similarly, $\left\Vert\left. \tilde{\nabla} g(x) \right\Vert\right. \leq C_{g} $ for all  $ x \in X.$
\end{remark}
 In the following lemma, we provide the bound on sequence $\{x^*_{\eta_{{k}}}\}$, which is solution of \eqref{eta_problem}.
 \begin{lemma} {\bf (Bound on $x^*_{\eta_k}$, see Proposition 1 of \cite{Yousefian}): }\label{final_convergence_lemma}
 	Consider problem \eqref{Q} and \eqref{eta_problem}. Suppose Assumption \ref{obj_func} holds. Then for a sequence \{$x^*_{\eta_k}$\}, and $x^*_g$ for any $k \geq 1$, the following hold,\\
(a) $ \left\Vert\left.  x^*_{\eta_k} -  x^*_{\eta_{k-1}} \right\Vert\right. \leq \frac{C_g}{\mu} \left| \frac{\eta_{k-1}}{\eta_k}  -1 \right|$.
 		(b) When  $\{\eta_{{k}}\}$ goes to zero,  $\{x^*_{\eta_{{k}}}\}$ converges to $x^*_g$. 
 \end{lemma}
 
 Our  objective is to show 
$
 \| x_{k+1}- x_g^*\| \rightarrow 0.  
$
Now from the triangle inequality,  
$
 \| x_{k+1}- x_{\eta_{k}}^*\| \rightarrow0$ and $\| x_{\eta_{k}}^*- x_g^*\| \rightarrow 0.
$
We know $\| x_{\eta_{k}}^*- x_g^*\|\rightarrow 0$ as $\eta_k\rightarrow0$.  Our main objective is to show $\| x_{k+1}- x_{\eta_{k}}^*\| \rightarrow 0.$ Next we define an error function which will be used in the convergence analysis. 
 
 
 \begin{definition} \label{function_L}
 	Let Assumption \ref{random sample} hold. 
 	Then for any $x,y \in \mathbb{R}^n$, function $\mathcal{L}(x,y) =  \sum_{i=1}^{d} \prob{i}^{-1} \left\Vert\left. x^{(i)}-y^{(i)} \right\Vert\right.^2$.
 \end{definition}
The following corollary holds from Definition \ref{function_L}.
 \begin{corollary} \label{Bounds_L}
 	Consider Definition \ref{function_L}, $\prob{{max}}$ and $\prob{{min}}$  as defined in the notation, and let Assumption \ref{random sample} hold. Then for any $x,y \in \mathbb{R}^n$, 
 	$
 \prob{{max}}	{\mathcal{L}(x,y)} \leq \|x-y\|^2 \leq\prob{{min}} {\mathcal{L}(x,y)}.
 	$
 \end{corollary}
%

\section{Convergence analysis of \ref{algorithm_proposed} scheme}\label{sec:convergence}
Here we begin with deriving a recursive error bound, that will be used later to show the convergence. 	
 \begin{lemma}  {\bf(Recursive relation for $\mathcal{L}\left( x_{k+1},  x_{\eta_k}^*\right)$):}  \label{lemma_5}
 	Consider problem \eqref{Q} and \eqref{eta_problem}. Let Assumptions \ref{obj_func} and \ref{random sample} hold. Let $\{x_k\}$ be the sequence generated from Algorithm \ref{algorithm_1}. Let positive sequences $\{\gamma_k\}$, and $\{\eta_k\}$ be non-increasing and $ \gamma_0\,\eta_0<{1}/{\mu\prob{{min}}} $. Then the following relation holds,\
 	\begin{align*} 
 	\EXP{\mathcal{L} \left(x_{k+1}, x^*_{\eta_k}\right)|\mathcal{F}_k} \leq  \left(1-{\mu\gamma_k\eta_k\prob{{min}}} \right)\mathcal{L} \left(x_{k}, x^*_{\eta_{k-1}}\right)\nonumber\\
 	+ \frac{2 C_g^2}{\prob{{min}}^2\mu^3\gamma_{{k}}\eta_{{k}}}\left( \frac{\eta_{k-1}}{\eta_k} - 1 \right)^2+ 2\gamma^2_k(C^2_f + \eta_0^2C_g^2).
 	\end{align*}
 \end{lemma}
 \begin{proof}
 	Consider $\mathcal{L} (x_{k+1}, x^*_{\eta_k})$. From the Definition \ref{function_L},
 	\begin{align}
 	&\mathcal{L} \left(x_{k+1}, x^*_{\eta_k}\right)  =  \sum_{i=1}^{d}  \prob{i}^{-1} \left\Vert\left.  x_{k+1}^{(i)}-  x_{\eta_{k}}^{*^{(i)}} \right\Vert\right.^2 \nonumber=\\
 	& \sum_{i=1, \thinspace	 i \neq i_k}^{d} \prob{i}^{-1} \left\Vert\left.  x_{k}^{(i)}-  x_{\eta_{k}}^{*^{(i)}} \right\Vert\right.^2  + \prob{{i_k}}^{-1} \underbrace{\left\Vert\left.  x_{k+1}^{(i_k)}-  x_{\eta_{k}}^{*^{(i_k)}} \right\Vert\right.^2}_{\text{term-1}} \label{eqn_12}
 	\end{align}
Since $x^*_{\eta_k} \in X$, we have $x^{*^{(i_k)}}_{\eta_k} \in X_{i_k}$. Now from the non-expansive property of  projection operator, term-1 becomes,
 	\begin{align*} 
 	&\left\Vert\left. x_{k+1}^{(i_k)}-  x_{\eta_{k}}^{*^{(i_k)}}\right\Vert\right.^2 \leq\\&\left\Vert\left.  x_{k}^{(i_k)}-\gamma_{k}\left(\tilde{\nabla}_{i_k} f \left( x_{k}\right) + \eta_{k} \tilde{\nabla}_{i_k} g\left( x_{k}\right)\right) -  x_{\eta_{k}}^{*^{(i_k)}} \right\Vert\right.^2.
 	\end{align*}
 	From the two preceding relations, we have,  
 	\begin{align} \label{Eq_17}
 & 	\mathcal{L}\left(x_{k+1}, x^*_{\eta_k}\right)\nonumber \\&= \sum_{i=1, \thinspace	 i \neq i_k}^{d} \prob{i}^{-1} \left\Vert\left.  x_{k}^{(i)}-  x_{\eta_{k}}^{*^{(i)}} \right\Vert\right.^2 +  \prob{{i_k}}^{-1}\left\Vert\left.  x_{k}^{(i_k)}-  x_{\eta_{k}}^{*^{(i_k)}} \right\Vert\right.^2 \nonumber \\& - 2\,\prob{{i_k}}^{-1}\gamma_k\left( x^{(i_k)}_{k} - x^{*^{(i_k)}}_{\eta_k}\right)^T\left(\tilde{\nabla}_{i_k} f\left( x_{k}\right) + \eta_{k} \tilde{\nabla}_{i_k} g\left( x_{k}\right)\right)\nonumber \\ &+ \prob{{i_k}}^{-1} \underbrace{\gamma_k^2 {\left\Vert\left.\tilde{\nabla}_{i_k} f\left( x_{k}\right)  +\eta_{k} \tilde{\nabla}_{i_k} g\left(x_{k}\right) \right\Vert\right.^2}}_{\text{term-2}}.
 	\end{align}
 	From Assumptions \ref{obj_func} (d) and Remark \ref{remark_subgrad_bound}, 
term-2 =
 	$\gamma_k^2 \left\Vert\left.\tilde{\nabla}_{i_k} f( x_{k}) + \eta_{k} \tilde{\nabla}_{i_k} g( x_{k}) \right\Vert\right.^2 \leq 2\gamma^2_kC_{f,i_k}^{2} + 2\gamma^2_k\eta_k^2C_{g,i_k}^{2}. $
 	Thus from \eqref{Eq_17}, and Definition \ref{function_L}, we obtain,
 	$
 	\mathcal{L} \left(x_{k+1}, x^*_{\eta_k}\right)\leq\mathcal{L} \left(x_{k}, x^*_{\eta_k}\right) +\prob{{i_k}}^{-1} 2\gamma^2_k \left(C_{f,i_k}^{2} + \eta_k^2C_{g,i_k}^{2}\right) - 2 \prob{{i_k}}^{-1}\gamma_k\left( x^{(i_k)}_{k} - x^{*^{(i_k)}}_{\eta_k}\right)^T\left(\tilde{\nabla}_{i_k} f\left( x_{k}\right) + \eta_{k} \tilde{\nabla}_{i_k} g\left( x_{k}\right)\right).$
 	Now taking the conditional expectation on both the sides, and taking into account $\mathcal{L}\left(x_k,x^*_{\eta_{{k}}}\right)$ is $\mathcal{F}_k$ measurable, \\
 	$\EXP{\mathcal{L} \left(x_{k+1}, x^*_{\eta_k}\right)|\mathcal{F}_k}\leq\\\mathcal{L} \left(x_{k}, x^*_{\eta_k}\right) +  2\gamma^2_k \underbrace{\EXP{\prob{{i_k}}^{-1}  C_{f,i_k}^{2}|\mathcal{F}_k}}_{\text{term-3}}\nonumber + 2\gamma_k^2\eta_k^2\underbrace{\EXP{\prob{{i_k}}^{-1}C_{g,i_k}^{2}|\mathcal{F}_k}}_{\text{term-4}} \nonumber-\\ 2\gamma_k\underbrace{\EXP{\prob{{i_k}}^{-1}\hspace{-1.5mm} \left( x^{(i_k)}_{k} - x^{*^{(i_k)}}_{\eta_k}\right)^T\hspace{-2.5mm}\left(\tilde{\nabla}_{i_k} f\left( x_{k}\right) + \eta_{k} \tilde{\nabla}_{i_k} g\left( x_{k}\right)\right)|\mathcal{F}_k}}_{\text{term-5}}.$
 	term-3 = $\sum_{i=1}^{d} C_{f,i}^2 = C_{f}^{2}$,  term-4 $  = C_{g}^{2}.
 	$
 	Also, term-5=    
 	$
 	 \sum_{i = 1}^{d} \prob{i} \left(\prob{{i}}^{-1} \left( x^{(i)}_{k} - x^{*^{(i)}}_{\eta_k}\right)^T\left(\tilde{\nabla}_{i} f\left( x_{k}\right) + \eta_{k} \tilde{\nabla}_{i} g\left( x_{k}\right)\right)\right)\\
 	 =\left( x_{k} - x^{*}_{\eta_k}\right)^T\left(\tilde{\nabla} f\left( x_{k}\right) + \eta_{k} \tilde{\nabla} g\left( x_{k}\right)\right).$
 	Substituting the values of term-3, term-4 and term-5, we obtain,
 	\begin{align} \label{Eq_21}
 	&\EXP{\mathcal{L} \left(x_{k+1}, x^*_{\eta_k}\right)|\mathcal{F}_k}=\mathcal{L} \left(x_{k}, x^*_{\eta_k}\right) + 2\gamma^2_k C_f^2 + 2\gamma^2_k\eta_k^2C_g^2 \nonumber\\ & - 2\,\gamma_k\left( x_{k} - x^{*}_{\eta_k}\right)^T\left(\tilde{\nabla} f\left( x_{k}\right) + \eta_{k} \tilde{\nabla} g\left( x_{k}\right)\right).
 	\end{align}
  Now from Remark \ref{convexity_subgradients}, for $x_1=x_k$ and $x_2=x^*_{\eta_{{k}}}$,  we have, 
 	\begin{align}\label{equation_5}
 	&\left (\tilde{\nabla} f\left( x_k\right)+ \eta_k \tilde{\nabla} g\left( x_k\right)\right)^T\left( x_k- x^{*}_{\eta_k}\right) - \left(\tilde{\nabla} f\left( x^{*}_{\eta_k}\right)\right. \nonumber\\ & \left.+ \eta_k \tilde{\nabla} g\left( x^{*}_{\eta_k}\right)\right)^T\left( x_k- x^{*}_{\eta_k}\right) \geq \eta_k \mu\left\Vert\left. x_k- x^{*}_{\eta_k}\right\Vert\right.^2.
 	\end{align} 
 	From the optimality conditions on  (\ref{eta_problem}), we have,
 	$\left(\tilde{\nabla} f\left( x^{*}_{\eta_k}\right)+ \eta_k \tilde{\nabla} g\left( x^{*}_{\eta_k}\right)\right)^T\left( x_k- x^{*}_{\eta_k}\right) \geq 0. $
 	Thus,\\
 	$ \left (\tilde{\nabla} f\left( x_k\right)+ \eta_k \tilde{\nabla} g\left( x_k\right)\right)^T\left( x_k- x^{*}_{\eta_k}\right) \geq \eta_k \mu\left\Vert\left. x_k- x^{*}_{\eta_k}\right\Vert\right.^2$ \\
 	Now, from  (\ref{Eq_21}) and the preceding inequality, we can write,
 \begin{align*}
 	\EXP{\mathcal{L} \left(x_{k+1}, x^*_{\eta_k}\right)|\mathcal{F}_k}\leq&\mathcal{L} \left(x_{k}, x^*_{\eta_k}\right)  - 2 \gamma_k \eta_k \mu \underbrace{\left\Vert\left. x_k- x^{*}_{\eta_k}\right\Vert\right.^2}_{\text{term-6}}\\&+ 2\gamma^2_k C_f^2 + 2\gamma^2_k\eta_k^2C_g^2.
\end{align*}
 	From Corollary \ref{Bounds_L}, bounding term-6, we have,
 	\begin{align} \label{Eq_26}
 	\EXP{\mathcal{L} \left(x_{k+1}, x^*_{\eta_k}\right)|\mathcal{F}_k} \leq &\left(1-2 \gamma_k \eta_k \mu \prob{{min}} \right) \mathcal{L} \left(x_{k}, x^*_{\eta_k}\right) \nonumber \\& + 2\gamma^2_k C_f^2  +  2\gamma^2_k\eta_k^2C_g^2.
 	\end{align} 
 	Now consider $\left\Vert\left. x_k - x^{*}_{\eta_k} \right\Vert\right.^2$. It can be written as,
 	\begin{align} \label{convexity_final_2}
 	&\left\Vert\left. x_k - x^*_{\eta_k} \right\Vert\right.^2 = \left\Vert\left. x_k - x^*_{\eta_{k-1}} \right\Vert\right.^2 + {\left\Vert\left. x^*_{\eta_{k-1}} - x^*_{\eta_k} \right\Vert\right.^2} \nonumber\\& + \underbrace{2 (x_k - x^*_{\eta_{k-1}})^T(x^*_{\eta_{k-1}}-x^*_{\eta_k})}_{\text{term-7}}.
 	\end{align}
 $c\in\mathbb{R}^n$, term-7 $ \leq \left(c \left\Vert x_k - x^*_{\eta_{k-1}} \right\Vert\right)^2+ \left(\frac{\left\Vert x^*_{\eta_{k-1}}-x^*_{\eta_k}	\right\Vert}{c}\right)^2.$\\ 
 	 Substituting above in equation \eqref{convexity_final_2}, with $c = \sqrt{\prob{{min}}\mu \gamma_k \eta_k}$, 
 		$
 	\left\Vert\left. x_k - x^*_{\eta_k} \right\Vert\right.^2 \leq \left( 1 + \prob{{min}}\mu \gamma_k \eta_k \right) \left\Vert\left. x_k - x^*_{\eta_{k-1}} \right\Vert\right.^2 + \left( 1 + \frac{1}{\prob{{min}}\mu \gamma_k \eta_k}  \right) \left\Vert\left. x^*_{\eta_{k-1}} - x^*_{\eta_k} \right\Vert\right.^2 .
 	$\\
	From Lemma \ref{final_convergence_lemma}, and
 	Corollary (\ref{Bounds_L}), we obtain,\\
$
 	\prob{{min}} \,\mathcal{L} \left(x_{k}, x^*_{\eta_k}\right) \leq \left( 1 + \prob{{min}} \mu \gamma_k \eta_k \right)   
 	\prob{{max}} \, \mathcal{L} \left(x_{k}, x^*_{\eta_{k-1}}\right)+ \left( 1 + \frac{1}{\prob{{min}}\mu \gamma_k \eta_k}  \right) \frac{C_g^2}{\mu^2}\left( \frac{\eta_{k-1}}{\eta_k} - 1 \right)^2.$\\
 	 	Dividing both sides of previous inequality by $\prob{{min}}$, 
and substituting this in  (\ref{Eq_26}), we obtain the following,
 	\begin{align*}
 	&\EXP{\mathcal{L} \left(x_{k+1}, x^*_{\eta_k}\right)|\mathcal{F}_k}\leq1 2\gamma^2_k C_f^2   +  2\gamma^2_k\eta_k^2C_g^2\\ 
	 & \frac{\prob{{max}}}{\prob{{min}}} \underbrace{ \left(1-2 \gamma_k \eta_k \mu \prob{{min}} \right) 
 		\left( 1 +\prob{{min}} \mu \gamma_k \eta_k \right) }_{\text{term-9}} \mathcal{L} \left(x_{k}, x^*_{\eta_{k-1}}\right)\\
 	 	&-2 \gamma_k \eta_k \mu \prob{{min}} \frac{C_g^2}{\mu^2\,\prob{{min}}}\left( 1 + \frac{1}{\prob{{min}} \mu \gamma_k \eta_k}  \right) \left( \frac{\eta_{k-1}}{\eta_k} - 1 \right)^2\\
 	&+\frac{C_g^2}{\mu^2\,\prob{{min}}}\left( 1 + \frac{1}{\prob{{min}} \mu \gamma_k \eta_k}  \right) \left( \frac{\eta_{k-1}}{\eta_k} - 1 \right)^2.
 	\end{align*}
 	We have, term-9 $\leq 1-\mu\gamma_k\eta_k \prob{{min}}$, now we can write,
 	\begin{align*}
 	&\EXP{\mathcal{L} \left(x_{k+1}, x^*_{\eta_k}\right)|\mathcal{F}_k} \leq   2\gamma^2_k C_f^2 \\ &+\frac{\prob{{max}}}{\prob{{min}}} (1-\mu\gamma_k\eta_k\prob{{min}} )\mathcal{L} \left(x_{k}, x^*_{\eta_{k-1}}\right)
 	+  2\gamma^2_k\eta_k^2C_g^2+\\&\underbrace{ \frac{\left(1-2 \gamma_k \eta_k \mu \prob{{min}}\right) C_g^2}{\mu^2\,\prob{{min}}}\left( 1 + \frac{1}{\prob{{min}} \mu \gamma_k \eta_k}  \right) \left( \frac{\eta_{k-1}}{\eta_k} - 1 \right)^2}_{\text{term-10}}.
 	\end{align*}
 	We have ${\gamma_{_0}}{\eta_{_0}} < \frac{d}{\prob{{min}}\mu} $, Bounding term-10, we have,
 	\begin{align*}
 	\EXP{\mathcal{L} \left(x_{k+1}, x^*_{\eta_k}\right)|\mathcal{F}_k} \leq  \left(1-{\mu\gamma_k\eta_k\prob{{min}}} \right)\mathcal{L} \left(x_{k}, x^*_{\eta_{k-1}}\right)
 	+ \\\frac{C_g^2}{\prob{{min}}\mu^2}\left(\frac{2 }{ \prob{{min}}\mu \gamma_k \eta_k}  \right) \left( \frac{\eta_{k-1}}{\eta_k} - 1 \right)^2\nonumber 
 	 + 2\gamma^2_k C_f^2   +  2\gamma^2_k\eta_k^2C_g^2.
 	 \end{align*}
Bounding non-increasing sequence, $\eta_k$ we get the result.
 \end{proof}

 \subsection{Convergence analysis}
 
\begin{remark}
	Throughout the analysis, we assume that blocks are randomly selected using a uniform distribution. 
\end{remark}
 
 \begin{assumption}\label{assum:IRLSA} Let the following hold: 
 	\begin{itemize}
 		\item [(a)] $\{\gamma_k\} \text{ and } \{\eta_k \}$ are positive sequences for $k \geq 0$ converging to zero such that $ {\gamma_{_0}}{\eta_{_0}} < \frac{d}{\mu}$;
 		\item [(b)]  $ \sum_{k=0}^{\infty}\hspace{-0.1cm}{\gamma_k\eta_k} = \infty;$  (c) $ \sum_{k=0}^{\infty}\hspace{-0.1cm}{\left(\frac{1}{\gamma_k \eta_k}  \right) \left( \frac{\eta_{k-1}}{\eta_k} - 1 \right)^2} < \infty;$ 
 		\item [(d)] $\sum_{k=0}^{\infty} \gamma_k^2 < \infty;$ (e) $\lim_{k\to\infty} \left(\frac{1}{ \gamma_k^2 \eta_k^2}  \right) \left( \frac{\eta_{k-1}}{\eta_k} - 1 \right)^2 = 0; $ \item[(f)]$ \lim_{k\to\infty} \frac{\gamma_k }{\eta_k} = 0$.
 	\end{itemize}
 \end{assumption}
Next, we show the a.s. convergence of the sequence $\{x_k\}$.
  \begin{proposition}{\bf (a.s. convergence of $\{x_k\}$):} \label{proposition_1}Consider  \eqref{Q} and \eqref{eta_problem}. Let Assumption \ref{assum:IRLSA} hold. Consider the sequence \{$x_k$\} is obtained by Algorithm \ref{algorithm_1}, and the sequence $\{x_{\eta_k}^*\}$ suppose obtained by solving  \eqref{eta_problem}. Then, 	
 		   $\mathcal{L}\left(x_k, x_{\eta_{k-1}}^*\right)$ goes to zero a.s.  and  $ \lim\limits_{k\rightarrow \infty} \EXP{\mathcal{L}\left(x_k, x_{\eta_{k-1}}^*\right)} =0 .$\\
 \end{proposition}
 \begin{proof}
 	We apply Lemma \ref{Lemma_convergence_2} to the result of Lemma \ref{lemma_5}.  
 		$v_k \triangleq \mathcal{L} \left(x_{k}, x^*_{\eta_{k-1}}\right)$, $ \alpha_k \triangleq \frac{\mu\gamma_k\eta_k}{d}$, $ \beta_k \triangleq  \left( \frac{2 d^2}{ \mu \gamma_k \eta_k}  \right) \frac{C_g^2}{\mu^2}\left( \frac{\eta_{k-1}}{\eta_k} - 1 \right)^2 + 2\gamma^2_k(C^2_f + \eta_0^2C_g^2)$.
 	Now, in order to claim the convergence of $v_k $, we show that all conditions of Lemma \ref{Lemma_convergence_2} hold. Note that $ \prob{i}= 1/d. $ From Assumption \ref{assum:IRLSA} (a), definition of \{$\gamma_k$\}, \{$\eta_k$\}, and from $ \gamma_{_0}\eta_{_0}<\frac{d}{\mu} $, the first condition of Lemma \ref{Lemma_convergence_2} is satisfied. Now consider sequence $\beta_k$. From Assumption \ref{assum:IRLSA} (a), sequences $\{\gamma_k\}$, $ \{\eta_k\} \text{ and the constant } \mu$ are positive, so the second condition of Lemma \ref{Lemma_convergence_2} is satisfied. Now in $  \sum_{k=0}^{\infty} \alpha_k, \text{ i.e. } \sum_{k=0}^{\infty} \frac{ \mu \gamma_k \eta_k} {d}$.  From  Assumption \ref{assum:IRLSA}(b), the third condition of Lemma \ref{Lemma_convergence_2} holds. Now from the definition of $ \beta_k$ and 
 	from Assumption \ref{assum:IRLSA}(c) and (d), the fourth condition of Lemma \ref{Lemma_convergence_2} holds. Finally consider $ \lim_{k\rightarrow \infty} \left(\frac{\beta_k}{\alpha_k}\right) = 0 $. Using the definition of $\beta_k$ and Assumption \ref{assum:IRLSA}(e, f), condition 5 of Lemma \ref{Lemma_convergence_2} holds. Thus  we get the required result.
 \end{proof}

 Next in Lemma \ref{sequence_parameters_&_gamma_eta} we give the choice of sequences $ \gamma_k \text{ and } \eta_k$  that satisfy Assumption \ref{assum:IRLSA}.
 
 \begin{lemma} \label{sequence_parameters_&_gamma_eta}
 	Let Assumption \ref{random sample} hold. Then  sequences $\{\gamma_k\}$ and $\{ \eta_k\}$  given by $\gamma_k = \gamma_0(k+1)^{-a}$ and $\eta_k = \eta_0(k+1)^{-b}$ where a, and b satisfy,
$
 	a>0, \quad b >0, \quad a+b < 1, \quad b < a, \quad a> 0.5, 
$
 	where $\gamma_0>0$ and $\eta_0>0$. Then $\{\gamma_k\}$ and $\{\eta_k\}$ satisfy Assumption \ref{assum:IRLSA}. 
 \end{lemma}  
 
 \begin{proof}
Similar to the proof of Lemma 5 in \cite{Yousefian}. Omitted because of the space requirements.  
 \end{proof}
Next, we show the a.s. convergence of the sequence $\{\bar{x}_k\}$.
  \begin{theorem}{\bf (a.s. convergence of $\{\bar{x}_k\}$): } \label{convergence_AS_mean_SENSES} Consider problem \eqref{Q}. Let $\gamma_{{k}} \text{ and } \eta_{{k}}$ be the sequences defined by Lemma \ref{sequence_parameters_&_gamma_eta}  	where $\gamma_0>0$, $\eta_0>0$, and $ ar< 1,$. Then $\{\bar{x}_k\}$ converges to the unique solution of \eqref{Q}, $x^*_g$ a.s.
 \end{theorem}
 \begin{proof}
  From $\lambda_{t,k} = {\gamma_t^r}/{\sum_{j=0}^{k} \gamma_j^r}$,
 	$
 	\left\Vert\left. \bar{x}_{k} - x^*_g \right\Vert\right. = \left\Vert\left. \sum_{t=0}^{k} \lambda_{t,k}x_t - \sum_{t=0}^{k} \lambda_{t,k}x^*_g \right\Vert\right. = \left\Vert\left. \sum_{t=0}^{k} \lambda_{t,k}\left(x_t - x^*_g\right) \right\Vert\right..
 	$
 	Using the triangle inequality,
 $
 	\left\Vert\left. \bar{x}_{k} - x^*_g \right\Vert\right. \leq \sum_{t=0}^{k}\lambda_{t,k} \left\Vert\left. x_t - x^*_g \right\Vert\right..	
 $
 	From definition of $\lambda_{t,k}$,
 	$
 	\left\Vert\left. \bar{x}_{k} - x^*_g \right\Vert\right. \leq \frac{\sum_{t=0}^{k}\gamma_t^r\left\Vert\left. x_t - x^*_g \right\Vert\right.}{\sum_{j=0}^{k} \gamma_j^r} .$
 	Comparing  with Lemma \ref{convergence_sum}, 
 $
 	\alpha_k \triangleq \gamma_k^r, \, u_k \triangleq \left\Vert\left. x_k - x^*_g \right\Vert\right., \,  v_{k+1} \triangleq \sum_{t=0}^{k}\gamma_t^r\left\Vert\left. x_t - x^*_g \right\Vert\right. .
$
 	Consider $ \sum_{k=0}^{\infty} \alpha_t $, i.e. $ \sum_{k=0}^{\infty} (1+k)^{-at}. $ As assumed in  Lemma \ref{convergence_AS_mean_SENSES}, we have $at<1$, so  $ \sum_{k=0}^{\infty} (1+k)^{-at} = \infty. $ From Proposition \ref{proposition_1} (b), $\left\Vert\left. x_k - x^*_g \right\Vert\right.\rightarrow 0$ a.s. Using Lemma \ref{convergence_sum}, we get the required result.
 \end{proof} 
  \section{Rate of convergence}
In this section, first  we derive the rate of convergence of \ref{algorithm_proposed} with respect to the inner level problem in \eqref{Q}.
%
  \begin{lemma}{\bf (Feasibility error bound for Algorithm \ref{algorithm_1})} \label{rate_bound}
 	Consider problem \eqref{Q} and $\{\bar{x}_k\}$, the sequence generated by Algorithm \ref{algorithm_1}. Let Assumption \ref{obj_func} hold, $r(<1)$ be an arbitrary scalar, and $\gamma_k$ be a non-increasing sequence. Let $\eta_{{k}}$ be a non-increasing sequence and $X$ to be bounded, i.e. $\|x\|\leq M$ for all $x \in X$ for some  $M>0$. Then for any $z \in X$, the following holds,
 	\begin{align*}
 	\EXP{f\left(\bar{x}_N\right)}-  f(z) \leq  \left(\sum_{i=0}^{N-1}\gamma_{{i}}^r\right)^{-1}\left(2 M_g \sum_{k=0}^{N-1}\gamma_{{k}}^r \eta_{{k}}+\right.\nonumber \\\left. 2\prob{{max}}M^2\left(\gamma_{{0}}^{r-1}+\gamma_{{N-1}}^{r-1}\right) + \left(\sum_{k=0}^{N-1} \gamma_{k}^{r+1}\right) \left(C_f^2 + C_g^2\eta_{{0}}^2\right)\right),\end{align*} where $M_g (>\hspace{-0.12cm}0)$ is a scalar such that $g(x) \hspace{-0.1cm}\leq \hspace{-0.12cm}M_g$ for all $x\in X.$
 \end{lemma}
 \begin{proof}
Consider equation \eqref{alg_avg} in step 6 of \ref{algorithm_proposed}. Note that using induction, it can be shown that $\bar{x}_k=\sum_{i=0}^{k}\lambda_{t,k} x_t,$ where $\lambda_{t,k} \triangleq \gamma_t^r/\sum_{j=0}^{k}\gamma_j^r. $
 	 
 	 Next, consider $\{ x_{k} \}$ be the sequence generated from Algorithm \ref{algorithm_1} and $z \in X$. Then from Definition \ref{function_L}, we have,
 	\begin{align*}
 	&\mathcal{L}\left(x_{k+1},z\right) = \\&\sum_{i = 1, i \neq i_k}^{d} \prob{i}^{-1} \left\Vert\left. x_k^{(i)} - z^{(i)} \right\Vert\right.^2 + \prob{{i_k}}^{-1}\underbrace{\left\Vert\left. x_k^{(i_k)} - z^{(i_k)} \right\Vert\right.^2}_{\text{term-1}}.
 	\end{align*}
 	Consider term-1. From  \ref{algorithm_proposed}, substituting  $x_{k+1}^{(i_k)}$ and using the non-expansiveness property of the projection operator,
 	\begin{align*}
 	&\left\Vert\left.x_{k+1}^{(i_k)} - z^{(i_k)}\right\Vert\right.^2 \\
 	&\leq \left\Vert\left.x_k^{(i_k)} - z^{(i_k)} \right\Vert\right.^2 + \gamma_k^2\left\Vert\left.\tilde{\nabla}_{i_k}f(x_k) + \eta_k \tilde{\nabla}_{i_k} g(x_k)\right\Vert\right.^2 \\&- 2 \gamma_{{k}}\left(x_k^{(i_k)} - z^{(i_k)} \right)^{T}\left(\tilde{\nabla}_{i_k}f(x_k) + \eta_k \tilde{\nabla}_{i_k} g(x_k)\right).
 	\end{align*}
 	Substituting the bound on term-1, we obtain,
	\begin{align*} 
 	&\mathcal{L}\left(x_{k+1},z\right) =  \\&\mathcal{L}\left(x_{k},z\right) +\prob{{i_k}}^{-1} \gamma_k^2\underbrace{\left\Vert\left.\tilde{\nabla}_{i_k}f(x_k) + \eta_k \tilde{\nabla}_{i_k} g(x_k)\right\Vert\right.^2}_{\text{term-2}}  \nonumber\\& -\prob{{i_k}}^{-1} 2  \gamma_{{k}}\left(x_k^{(i_k)} - z^{(i_k)} \right)^{T}\left(\tilde{\nabla}_{i_k}f(x_k) + \eta_k \tilde{\nabla}_{i_k} g(x_k)\right),
 	\end{align*}
 	here we used Definition \ref{function_L}. From Remark \ref{remark_subgrad_bound}, bounding term-2, 
 	\begin{align*}
 	\left\Vert\left.\tilde{\nabla}_{i_k}f(x_k) + \eta_k \tilde{\nabla}_{i_k} g(x_k)\right\Vert\right.^2 \leq  2 C_{f, i_k}^2 + 2 \eta_{{k}}^2 C_{g, i_k}^2.
 	\end{align*}
 	Substituting the bound of term-2, we get, 
 	\begin{align*} 
 	&\mathcal{L}\left(x_{k+1},z\right) \leq \mathcal{L}\left(x_{k},z\right) +2\prob{{i_k}}^{-1} \gamma_k^2   C_{f, i_k}^2+ 2 \prob{{i_k}}^{-1}  \eta_k^2 \gamma_k^2  C_{g, i_k}^2  \nonumber 	\\&-2\prob{{i_k}}^{-1}   \gamma_{{k}}\left(x_k^{(i_k)} - z^{(i_k)} \right)^{T}\left(\tilde{\nabla}_{i_k}f(x_k) + \eta_k \tilde{\nabla}_{i_k} g(x_k)\right).
 	\end{align*}
 	By taking conditional expectation on the both sides of equation above, and since  $\mathcal{L}\left(x_{k},z\right)$ is $\mathcal{F}_k$ measurable, 
 	\begin{align} \label{55}
 	&\EXP{\mathcal{L}\left(x_{k+1},z\right)|\mathcal{F}_k} \nonumber \leq \mathcal{L}\left(x_{k},z\right) +2   \gamma_k^2 \underbrace{\EXP{ \prob{{i_k}}^{-1} C_{f, {i_k}}^2|\mathcal{F}_k}}_{\text{term-3}}-2\gamma_k\\ &\underbrace{\EXP{\prob{{i_k}}^{-1}   \left(x_k^{(i_k)} - z^{(i_k)} \right)^{T}\left(\tilde{\nabla}_{i_k}f(x_k) + \eta_k \tilde{\nabla}_{i_k} g(x_k)\right)|\mathcal{F}_k}}_{\text{term-5}}\nonumber\\&+ 2 \eta_k^2 \gamma_k^2 \underbrace{ \EXP{ \prob{{i_k}}^{-1} C_{g, i_k}^2|\mathcal{F}_k} }_{\text{term-4}} .
 	\end{align} 
 	Using definition of expectation, $
 	 \text{term-3} = C_{f}^{2}$, term-4 $=  C_{g}^{2},
 	$
  \\ term-5    
 	$
 = \left( x_{k} - z\right)^T\left(\tilde{\nabla} f\left( x_{k}\right) + \eta_{k} \tilde{\nabla} g\left( x_{k}\right)\right).
 	$ From \eqref{55}, 	 
 	\begin{align} \label{56}
 	\EXP{\mathcal{L}\left(x_{k+1},z\right)|\mathcal{F}_k} \leq \mathcal{L}\left(x_{k},z\right) +2   \gamma_k^2 C_f^2+ 2 \eta_k^2 \gamma_k^2 C_g^2  \nonumber\\+2\gamma_k\underbrace{\left(z- x_{k} \right)^T\left(\tilde{\nabla} f\left( x_{k}\right) + \eta_{k} \tilde{\nabla} g\left( x_{k}\right)\right)}_{\text{term-6}}.
 	\end{align}
 	Using the definition of subgradient at point $x_k$, 
 	\begin{align*}\text{term-6} =
 &\left(z - x_{k}\right)^T\tilde{\nabla} f\left( x_{k}\right) + \eta_{{k}}\left(z - x_{k}\right)^T\tilde{\nabla} g\left( x_{k}\right) \\& \leq f(z) - f\left(x_k\right) + \eta_{{k}}g(z) - \eta_{{k}}g\left( x_k \right).
 \end{align*}
 	Bounding term-6, using conditional and total expectation,
 	\begin{align} \label{57}
 	&\EXP{\mathcal{L}\left(x_{k+1},z\right)} \leq \EXP{\mathcal{L}\left(x_{k},z\right)} +2   \gamma_k^2 C_f^2+ 2 \eta_k^2 \gamma_k^2 C_g^2 \nonumber\\& +2\gamma_k\left(f(z)+\eta_{{k}}g(z)-\EXP{f\left(x_k\right)+\eta_{{k}}g\left(x_k\right)}\right).
 	\end{align}
 	Multiplying the both sides of equation \eqref{57} by $\gamma_k^{r-1}$, and adding, subtracting $\gamma_{k-1}^{r-1}\EXP{\mathcal{L}\left(x_k, z\right)}$ on the left-hand side,
 	\begin{align}\label{58}
 	&\gamma_k^{r-1} \EXP{\mathcal{L}\left(x_{k+1}, z\right)} - \left(\gamma_k^{r-1}-\gamma_{k-1}^{r-1}\right)\underbrace{\EXP{\mathcal{L}\left(x_k,z\right)}}_{\text{term-7}}\nonumber\\&-\gamma_{k-1}^{r-1}\EXP{\mathcal{L}\left(x_k,z\right)}\nonumber\leq 2 \gamma_{k}^{r+1}C_f^2+ 2 \gamma_{k}^{r+1}\eta_{{k}}^2C_g^2 \nonumber  \\&+2 \gamma_k^r\left(f(z)+\eta_{{k}}g(z)-\EXP{f\left(x_k\right)+\eta_{{k}}g\left(x_k\right)}\right).
 	\end{align}
 	Since $r<1$ and $\gamma_{{k}}$ is a non-increasing, $\gamma_{{k-1}}^{r-1}-\gamma_{k}^{r-1}$ is a non-negative sequence. From Lemma \ref{Bounds_L},  $\mathcal{L}\left(x_k,z\right)\leq \prob{{max}}\|x_k-z\|^2\leq 2 \prob{{max}}\left(\|x_k\|^2+\|z\|^2\right)$. From the boundedness of set $X$,  $\EXP{\mathcal{L}\left(x_k,k\right)} \leq 4 \prob{{max}}M^2$. Substituting  bound on term-7 in \eqref{58} and summing up over $k =1, \dots, N-1$, 
 	\begin{align}\label{59}
 \nonumber	- \gamma_{0}^{r-1}&\EXP{\mathcal{L}\left(x_1,z\right)} -  4\gamma_{N-1}^{r-1}\prob{{max}}M^2  \leq 2 C_f^2 \sum_{k=1}^{N-1} \gamma_{k}^{r+1}\\&+ 2 C_g^2 \sum_{k=1}^{N-1}\gamma_{k}^{r+1}\eta_{{k}}^2 +2 \sum_{k=1}^{N-1}\gamma_k^r\left(f(z)+\eta_{{k}}g(z)\right)\nonumber\\&-2 \sum_{k=1}^{N-1}\gamma_k^r\EXP{f\left(x_k\right)+\eta_{{k}}g\left(x_k\right)}.
 	\end{align}	
 	putting  $k=0$ in  \eqref{57}, 
 	$ 
 	\EXP{\mathcal{L}\left(x_{1},z\right)} \leq \underbrace{\EXP{\mathcal{L}\left(x_{0},z\right)}}_{\text{term-8}} +2   \gamma_0^2 C_f^2+ 2 \eta_0^2 \gamma_0^2 C_g^2  +2\gamma_0\left(f(z)+\eta_{{0}}g(z)-\EXP{f\left(x_0\right)+\eta_{{0}}g\left(x_0\right)}\right).
 	$\\
 	Now, term-8 $
 	\leq 4 \prob{{max}}M^2 +2   \gamma_0^2 C_f^2+ 2 \eta_0^2 \gamma_0^2 C_g^2  +2\gamma_0\left(f(z)+\eta_{{0}}g(z)-\EXP{f\left(x_0\right)+\eta_{{0}}g\left(x_0\right)}\right).
 	$\\
 	Multiplying the both sides of equation with $\gamma_{{0}}^{r-1}$, we get,
 	\begin{align} \label{60}
 	&\gamma_{{0}}^{r-1}\EXP{\mathcal{L}\left(x_{1},z\right)} - 4\gamma_{{0}}^{r-1}  \prob{{max}}M^2 \leq 2  \gamma_0^{r+1} C_f^2 +2 \eta_0^2 \gamma_0^{r+1} \nonumber\\&C_g^2  +2\gamma_0^r\left(f(z)+\eta_{{0}}g(z)-\EXP{f\left(x_0\right)+\eta_{{0}}g\left(x_0\right)}\right).
 	\end{align}
 	Adding  \eqref{59} and \eqref{60} together, and
 	combining the terms,
 	 	\begin{align*}
 	&-4 \prob{{max}}M^2\left(\gamma_{{0}}^{r-1}+\gamma_{{N-1}}^{r-1}\right) \leq 2 C_f^2\left(\sum_{k=0}^{N-1} \gamma_{k}^{r+1}\right)\\&+2C_g^2\left(\sum_{k=0}^{N-1}\gamma_{k}^{r+1}\eta_{{k}}^2\right)+ 2\left(\sum_{k=0}^{N-1}\gamma_k^r\left(f(z)+\eta_{{k}}g(z)\right)\right) \\&-2\left(\sum_{k=0}^{N-1}\gamma_k^r\EXP{f\left(x_k\right)+\eta_{{k}}g\left(x_k\right)}\right).
\end{align*}
 	Dividing the both sides by $\sum_{i=0}^{N-1}\gamma_{{i}}^r$, and denoting $\frac{\gamma_k^r}{\sum_{i=0}^{N-1}\gamma_i^r} = \lambda_{k,N-1}$, we get,\\
 	$\
 	\underbrace{\sum_{k=0}^{N-1}\lambda_{k,N-1}\EXP{f\left(x_k\right)+\eta_{{k}}g\left(x_k\right)}}_{\text{term-9}} - \sum_{k=0}^{N-1}\lambda_{k,N-1} \left(f(z)\right.\\\left.+\eta_{{k}}g(z)\right) \leq \left(\sum_{i=0}^{N-1}\gamma_{{i}}^r\right)^{-1}\left(2\prob{{max}}M^2\left(\gamma_{{0}}^{r-1}+\gamma_{{N-1}}^{r-1}\right)\right.\\\left. + C_f^2\left(\sum_{k=0}^{N-1} \gamma_{k}^{r+1}\right) + C_g^2\left(\sum_{k=0}^{N-1}\gamma_{k}^{r+1}\eta_{{k}}^2\right)\right).
 	$\\
 	By updating term-9 and  rearranging the original terms,
 	\begin{align*}
 	\underbrace{\EXP{	\sum_{k=0}^{N-1}\lambda_{k,N-1}f\left(x_k\right)}}_{\text{term-10}} - \underbrace{\sum_{k=0}^{N-1}\lambda_{k,N-1} f(z)}_{\text{term-11}} \leq\nonumber\\ \sum_{k=0}^{N-1}\lambda_{k,N-1} \eta_{{k}}g(z) -  \EXP{	\sum_{k=0}^{N-1}\lambda_{k,N-1}\eta_{{k}}g\left(x_k\right)}\nonumber\\+ \nonumber\left(\sum_{i=0}^{N-1}\gamma_{{i}}^r\right)^{-1}\left(2\prob{{max}}M^2\left(\gamma_{{0}}^{r-1}+\gamma_{{N-1}}^{r-1}\right)\right.\\ \left. + C_f^2\left(\sum_{k=0}^{N-1} \gamma_{k}^{r+1}\right) + C_g^2\left(\sum_{k=0}^{N-1}\gamma_{k}^{r+1}\eta_{{k}}^2\right)\right).
 	\end{align*}
 	Using the convexity of $f$ and  the definition of $\lambda_{k,N-1}$, we have term-10 
 	$
 	 \leq \sum_{k=0}^{N-1}\lambda_{k,N-1}f\left(x_k\right) 
 	$, and term-11 $ =f(z).
 	$
 	\begin{align*}
 	&\EXP{f\left(\bar{x}_N\right)} -  f(z)\nonumber \leq \\ &\underbrace {\sum_{k=0}^{N-1}\lambda_{k,N-1} \eta_{{k}}g(z) -  \EXP{	\sum_{k=0}^{N-1}\lambda_{k,N-1}\eta_{{k}}g\left(x_k\right)}}_{\text{term-12}}\nonumber\\&+ \left(\sum_{i=0}^{N-1}\gamma_{{i}}^r\right)^{-1}\left(2\prob{{max}}M^2\left(\gamma_{{0}}^{r-1}+\gamma_{{N-1}}^{r-1}\right) \right.\nonumber\\&\left.+ C_f^2\left(\sum_{k=0}^{N-1} \gamma_{k}^{r+1}\right) + C_g^2\left(\sum_{k=0}^{N-1}\gamma_{k}^{r+1}\eta_{{k}}^2\right)\right).\end{align*}
	Using  definition of {$M_g$}, we obtain, \\ term-12= 
	$
 	\EXP{	\sum_{k=0}^{N-1}\lambda_{k,N-1} \eta_{{k}}g(z) - 	\sum_{k=0}^{N-1}\lambda_{k,N-1}\eta_{{k}}g\left(x_k\right)} \\ \leq \EXP{	\sum_{k=0}^{N-1}\lambda_{k,N-1} \eta_{{k}} |g(z) - g\left(x_k\right)|} \\ \leq 2 M_g \sum_{k=0}^{N-1}\lambda_{k,N-1}\eta_{{k}}.
 	$\\
 	Bounding term-12 and using the definition of $\lambda_{k,N-1}$, 
 	\begin{align*}
 	\EXP{f\left(\bar{x}_N\right)}-  f(z) \leq  \left(\sum_{i=0}^{N-1}\gamma_{{i}}^r\right)^{-1}\left(2 M_g \sum_{k=0}^{N-1}\gamma_{{k}}^r \eta_{{k}}+ C_f^2\right.\\\left.\sum_{k=0}^{N-1} \gamma_{k}^{r+1} \hspace{-0.25cm}+ 2\prob{{max}}M^2\left(\gamma_{{0}}^{r-1}+\gamma_{{N-1}}^{r-1}\right) + C_g^2\sum_{k=0}^{N-1}\gamma_{k}^{r+1}\eta_{{k}}^2\right)
 	\end{align*}
 	Here, since $\eta_k$ is a non-increasing sequence, bounding it by $\eta_{{0}}$, we get the required result.
 \end{proof}
  Next, we state Lemma \ref{rate_support} (see Lemma 9, pg. 418 of \cite{Yousefian}) and use it in Theorem \ref{rate_result} to derive the rate of convergence. 
   \begin{lemma} \label{rate_support}
 	For a scalar $\alpha \neq -1 $ and integers l, N, where $0\leq l \leq N-1$, we have 
  \begin{align*}
  \frac{N^{\alpha+1} - \left(l+1\right)^{\alpha+1}}{\alpha +1 } \leq \sum_{k=l}^{N-1}\left(k+1\right)^{\alpha}\leq \left(l+1\right)^{\alpha}\\ + \frac{\left( N+1 \right)^{\alpha+1}-\left(l+1\right)^{\alpha+1}}{\alpha+1}.
  \end{align*}
\end{lemma}
 In  Theorem \ref{rate_result}, we show the rate of convergence for \ref{algorithm_proposed}. 
 \begin{theorem}\label{rate_result}
 	Consider problem \eqref{Q} and the sequence generated from Algorithm \ref{algorithm_1} $\{\bar{x}_N\}$. Let Assumptions \ref{obj_func}, and \ref{random sample} hold. Let the sequence $\{ \gamma_{{k}} \} \text{ and } \{\eta_{{k}}\}$ are given by the following, 
 	$\gamma_{{k}} = {\gamma_{{0}}}/{\left(k+1\right)^{0.5+0.1 \delta}} $ and $ \eta_{{k}} = {\eta_{{0}}}/{\left(k+1\right)^{0.5-\delta}},$
 	such that $ \gamma_{{0}} >0, \ \eta_{{0}} >0, \ \gamma_{{0}}\eta_{{0}}<\frac{1}{\mu\prob{{}}},\ 0<\delta<0.5, \text{ and } r<1$. Then the following hold,\\
 	(i) Sequence $\{ \bar{x}_N \}$ converges to $x^*_{g}$ almost surely.
 	\\(ii) $\EXP{f\left( \bar{x}_N \right)}$ converges to the optimal solution of inner level of \eqref{Q}, $f^*$ with the rate of $\us{\cal O}\left({1}/{N^{0.5-\delta}}\right)$.
  \end{theorem}
 \begin{proof}	{ (i) Consider the sequences given for $\gamma_k \text{ and } \eta_{{k}}$. By denoting $a = 0.5+0.1 \delta \text{ and } b = 0.5-\delta$, we have, 
 		$
 		\gamma_{{k}} = {\gamma_{{0}}}/{\left(k+1\right)^{a}} $ and $ \eta_{{k}} = {\eta_{{0}}}/{\left(k+1\right)^{b}}.
 		$
 		Also we know that $0<\delta<0.5 \text{ and } r<1$. Therefore, we have:  $ a,b >0, \ b< a, \ 0.5<a<0.55, \ 0<b<0.5,\ a+b <1 \text{ and } ar<1$. So,  $\gamma_{{k}} \text{ and } \eta_{{k}}$ satisfy all the conditions of Lemma \ref{convergence_AS_mean_SENSES}. } \\
 	(ii) Substituting of $\gamma_{{k}}, \eta_{{k}}$, and   $x^*_g$ at the place of $z$, in Lemma \ref{rate_bound}, we obtain, 
 	\begin{align*}
 	&\EXP{f\left(\bar{x}_N\right)}-  f^* \nonumber \leq  \left(\gamma_{{0}}^r \sum_{i=0}^{N-1} \frac{1}{(k+1)^{ar}}\right)^{-1}\\&\left(2 \prob{{max}}M^2\gamma_{{0}}^{r-1}\left(N^{a(1-r)}+1\right)+\gamma_{{0}}^r\left(2M_g\eta_{{0}}\sum_{k=0}^{N-1}\right.\right.\\&\left.\frac{1}{(k+1)^{ar+b}}+\left.\left(C_f^2+C_g^2\eta_{{0}}^2\right)\gamma_{{0}}\underbrace{\sum_{k=0}^{N-1}\frac{1}{(k+1)^{ar+a}}}_{\text{term-1}}\right)\right).
 	\end{align*}
 	 modifying term-1 in equation above, and expanding terms,
 		 	 	  \begin{figure}[b]
 		 	 	  	\hspace{0.2cm} 
 		 	 	  	\begin{subfigure}[b]{0.5\linewidth}\vspace{-0.6cm}
 		 	 	  		\includegraphics[width=\linewidth]{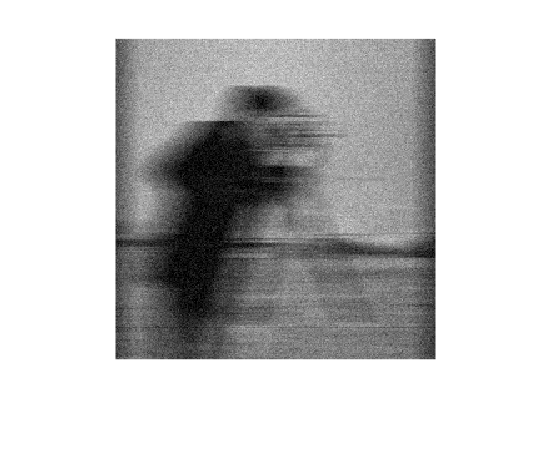}
 		 	 	  		\vspace*{-1.2cm}
 		 	 	  		\caption{Blurred image}
 		 	 	  	\end{subfigure}\hspace{0.34cm}
 		 	 	  	\begin{subfigure}{0.285\linewidth}
 		 	 	  		\vspace{-2.8cm}
 		 	 	  		\includegraphics[width=\linewidth]{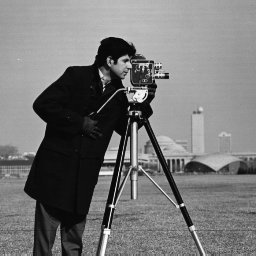}
 		 	 	  		\vspace*{-.45cm}
 		 	 	  		\caption{Original image}
 		 	 	  	\end{subfigure}
 		 	 	  	\caption{\footnotesize Blurred and original image of cameraman}
 		 	 	  	\label{blurred_original}
 		 	 	  	\vspace*{-0.5cm}
 		 	 	  \end{figure}
 		\begin{figure*}[t]
 		\centering
 		\begin{subfigure}[b]{0.23\linewidth}
 			\includegraphics[width=\linewidth]{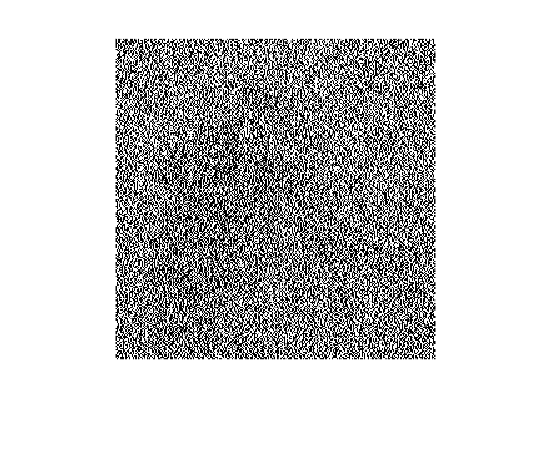}
 			\vspace*{-1.2cm}
 			\caption{$\eta$=0}
 		\end{subfigure} \hspace*{-1.cm}
 		\begin{subfigure}[b]{0.23\linewidth}
 			\includegraphics[width=\linewidth]{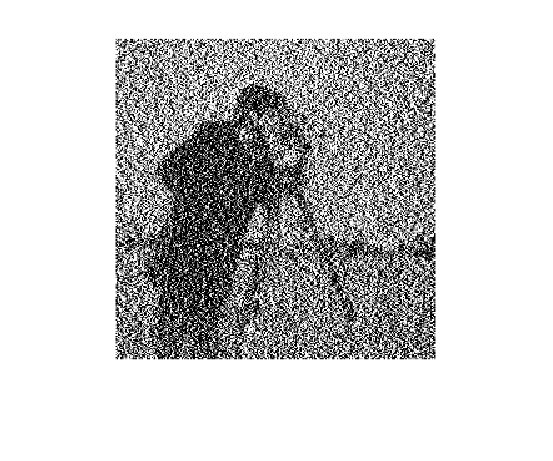}
 			\vspace*{-1.2cm}
 			\caption{$\eta$=0.001}
 		\end{subfigure}\hspace*{-1cm}
 		\begin{subfigure}[b]{0.23\linewidth}
 			\includegraphics[width=\linewidth]{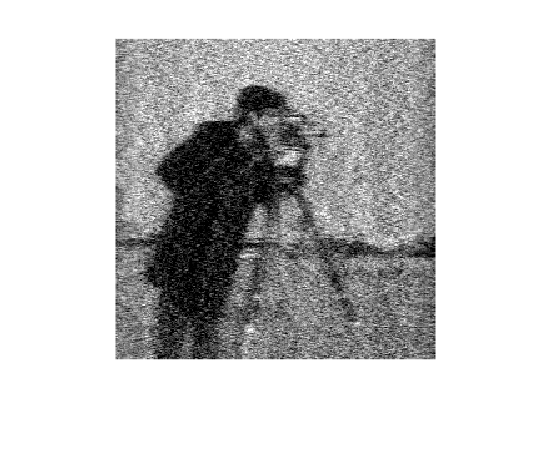}
 			\vspace*{-1.2cm}
 			\caption{$\eta$=0.01}
 		\end{subfigure}\hspace*{-1.cm}
 		\begin{subfigure}[b]{0.23\linewidth}
 			\includegraphics[width=\linewidth]{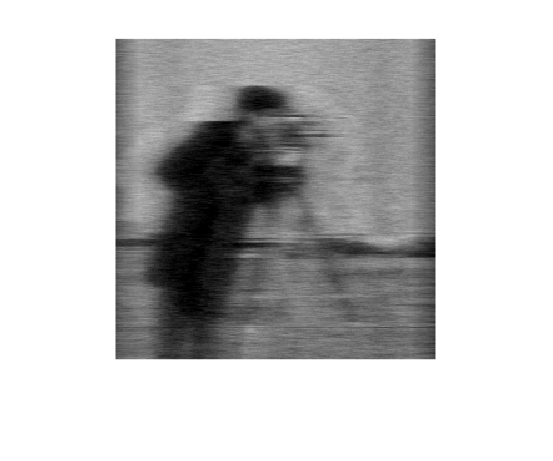}
 			\vspace*{-1.2cm}
 			\caption{$\eta$=0.1}
 		\end{subfigure}\hspace*{-1.cm}
 		\begin{subfigure}[b]{0.23\linewidth}
 			\includegraphics[width=\linewidth]{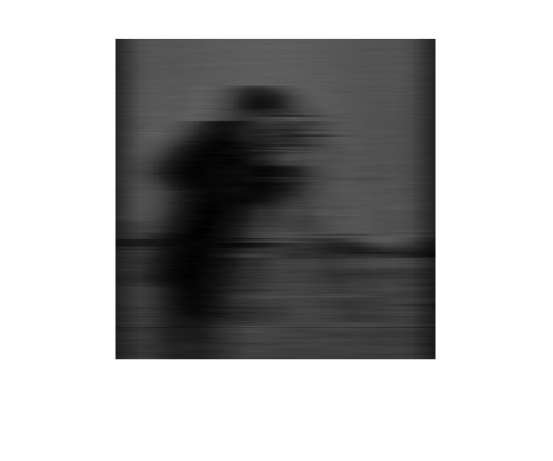}
 			\vspace*{-1.2cm}
 			\caption{$\eta$=1}
 		\end{subfigure}
 		\hspace*{\fill}\\ [1ex]
 		\begin{subfigure}[b]{0.23\linewidth}
 			\hspace{-0.25cm}\vspace{0.2cm}
 			\includegraphics[width=\linewidth]{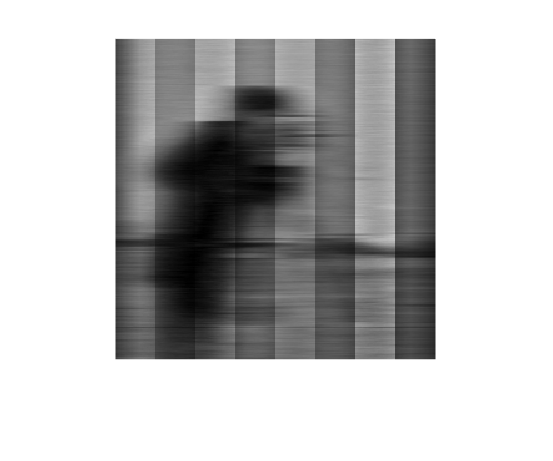}
 			\vspace*{-1.cm}
 			\caption{$k=10^2$}
 		\end{subfigure}\hspace*{-1cm}
 		\begin{subfigure}[b]{0.23\linewidth}
 			\includegraphics[width=\linewidth]{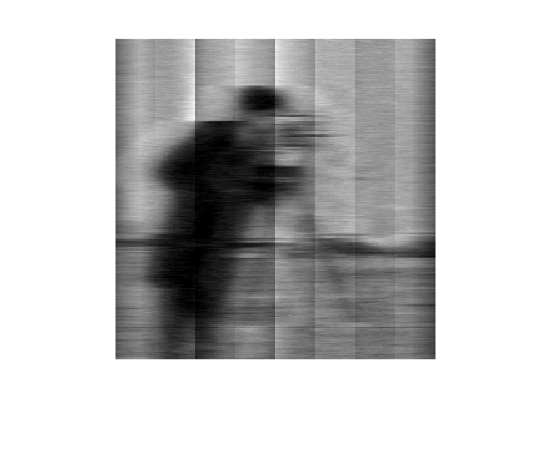}
 			\vspace*{-1.2cm}
 			\caption{$k=10^3$}
 		\end{subfigure}\hspace*{-1cm}
 		\begin{subfigure}[b]{0.23\linewidth}
 			\includegraphics[width=\linewidth]{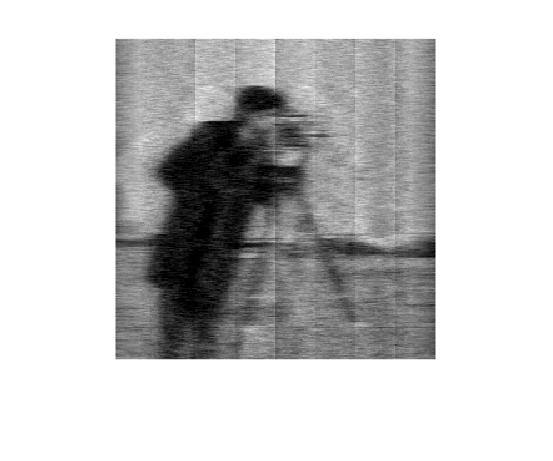}
 			\vspace*{-1.2cm}
 			\caption{$k=10^4$}
 		\end{subfigure}\hspace*{-1cm}
 		\begin{subfigure}[b]{0.23\linewidth}
 			\includegraphics[width=\linewidth]{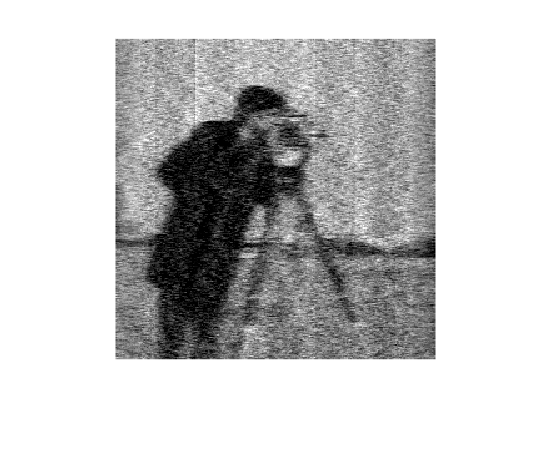}
 			\vspace*{-1.2cm}
 			\caption{$k=10^5$}
 		\end{subfigure}\hspace*{-1cm}
 		\begin{subfigure}[b]{0.23\linewidth}
 			\includegraphics[width=\linewidth]{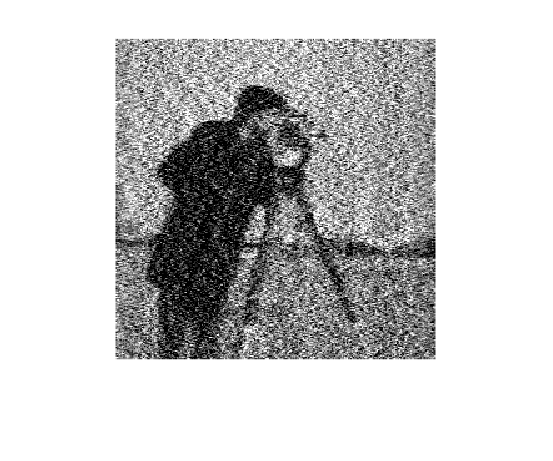}
 			\vspace*{-1.2cm}
 			\caption{$k=10^6$}
 		\end{subfigure} 
 		\caption{\footnotesize First row: Image deblurring using the regularization technique with different values of $\eta$, running for $10^5$ iterations. Second row: Image deblurring using \ref{algorithm_proposed}, stopping at different iterations $k$. 		\vspace*{-0.5cm}}	
 		\label{deblurring}
 	\end{figure*}
 
 $	\EXP{f\left(\bar{x}_N\right)}-  f^*\leq \\  2 \prob{{max}} M^2 \gamma_{{0}}^{-1}\left(\underbrace{\left( \sum_{i=0}^{N-1} \frac{1}{(k+1)^{ar}}\right)^{-1}N^{a(1-r)}}_{\text{term-3}}\right.\\\left.+\underbrace{\left( \sum_{i=0}^{N-1} \frac{1}{(k+1)^{ar}}\right)^{-1}}_{\text{term-2}}\right)  + \vphantom{\underbrace{\left( \sum_{i=0}^{N-1} \frac{1}{(k+1)^{ar}}\right)^{-1} \left( \sum_{k=0}^{N-1} \frac{1}{(k+1)^{ar+b}}\right)}_{\text{term-4}}}\left(2M_g\eta_{{0}}+\gamma_{{0}}\left(C_f^2+C_g^2\eta_{{0}}^2\right)\right).$
 	The above equation  can also be written as \\
	$
 	 	\EXP{f\left(\bar{x}_N\right)}-  f^* \leq 2 \prob{{max}} M^2 \gamma_{{0}}^{-1}\left({\text{term-3}}+{\text{term-2}}\right) +{\text{term-4}}\left(2M_g\eta_{{0}}+\gamma_{{0}}\left(C_f^2+C_g^2\eta_{{0}}^2\right)\right).
 	$\\
 	From Lemma \ref{rate_support}, we have,
 	term-2 $\leq \frac{1-ar}{N^{-ar+1}-1}$ = $\us{\cal O}\left(N^{-(1-ar)}\right),$
 	 term-3 $\leq \frac{(1-ar)N^{a(1-r)}}{N^{-ar+1}-1}$ = $\us{\cal O}\left(N^{-(1-a)}\right),$
 	 term-4 $ \leq \left(\frac{1-ar}{N^{-ar+1}-1}\right)\left(1+\frac{(N+1)^{1-(ar+b)}-1}{1-(ar+b)}\right)  = \us{\cal O}\left(N^{-(1-ar)}\right) + \us{\cal O}\left(N^{-b}\right).
 	$
 	Now, substituting bounds of terms-2, 3, and 4, we have,  \\
$ 	\EXP{f\left(\bar{x}_N\right)}-  f^* \leq \us{\cal O} \left(\mathsf{max} \Big \{N^{-(1-ar)}, N^{-(1-a)}, N^{-b} \Big \}\right) \nonumber\\= \us{\cal O} \left(N^{-\mathsf{min}\{1-ar, \ 1-a,\ b\}}\right).
$\\
 From definitions of $a, r, \text{ and } \delta,$  we obtain the result. \end{proof}
\section{Application of \ref{algorithm_proposed}}
One of the ways to address the ill-posedness in image deblurring is employing the regularization. The ill-posed problem \eqref{ImageProc} is converted into the   regularized problem \eqref{eta_problem} by substituting  functions $f(x) = \|b-Ax\|^2$, and $g(x) = \|x\|_2^2$ in \eqref{eta_problem}.  
As the value of regularization parameter $\eta$ changes,  we solve a different optimization problem \eqref{eta_problem}. The basic idea is, $\eta (\in (0, +\infty))$ governs the way by which solutions of linear inverse problem \eqref{ImageProc} are approximated by \eqref{eta_problem}.  
%
%

We are provided with the blurred noisy image Fig. \ref{blurred_original}(a), which is further converted into the column vector $b$. Our objective is to get the original image, Fig. \ref{blurred_original} (a) using image deblurring. Here we compare two ways of deblurring: standard regularization, and  \ref{algorithm_proposed}.

\noindent{\bf Inference: }  Fig. \ref{deblurring}(a)--(e) show the deblurred images obtained by conventional regularization at different  $\eta$ for $10^5$ iterations.  Fig. \ref{deblurring}(f)--(j) show the deblurred images  using \ref{algorithm_proposed} with stopping at different iteration. \ref{algorithm_proposed} is computationally effective because unlike as the case of conventional regularization, in \ref{algorithm_proposed} we solve the problem instance just once. The tricky part is at what iteration $k$ we should  stop. Stopping at a suitable iteration $k$ is desired because that governs the deblurred image quality. Practically (using \ref{algorithm_proposed}), this seems to be feasible because we could save images after a regular interval of iterations and would stop at any iteration $k$ when the deblurred picture is good enough.
\section{Conclusion} 
In this work, we consider a bilevel optimization problem \eqref{Q} with high dimensional solution space.   Random block coordinate iterative regularized gradient descent \eqref{algorithm_proposed} scheme is developed to address problem \eqref{Q}. We establish the convergence of sequence generated from \ref{algorithm_proposed} to the unique solution of \eqref{Q}. Furthermore, we derive the rate of convergence $ \centering \us{\cal O} \left(\frac{1}{{k}^{0.5-\delta}}\right)$, with respect to the inner level function of the bilevel problem. Our ground assumptions in the convergence proof and rate analysis are mild, such that  $f$ and $g$ can be nondifferentiable functions. 
Demonstration of \ref{algorithm_proposed} on image processing shows that our scheme computationally performs well  compared to the conventional {\it(two loop)} regularization schemes.


\end{document}